\documentclass[12pt]{amsart}
\usepackage[margin=1.2in]{geometry}
\usepackage[OT2,T1]{fontenc}
\usepackage[utf8]{inputenc}
\usepackage{tikz-cd,stmaryrd}
\usepackage{microtype,hyperref,amsfonts,amssymb,amsthm,mathrsfs,amsrefs}
\usepackage[cal=euler,scr=rsfs]{mathalfa}

\theoremstyle{plain}
\newtheorem{theorem}{Theorem}[section]
\newtheorem{proposition}[theorem]{Proposition}
\newtheorem{lemma}[theorem]{Lemma}

\newtheorem{corollary}[theorem]{Corollary}
\newtheorem*{conjecture}{Conjecture}
\theoremstyle{definition}
\newtheorem{definition}{Definition}[section]
\newtheorem{instance}[definition]{Example}

\newtheorem{remark}[definition]{Remark}
\usepackage{enumitem}
\setlist[itemize]{leftmargin=2em}
\numberwithin{equation}{section}
\title{Mirror symmetry for double cover Calabi--Yau varieties}
\date{\today}

\author[S.~Hosono]{Shinobu~Hosono}
\address{Shinobu~Hosono, Department of Mathematics, Gakushuin University, Mejiro, Toshima-ku, Tokyo 171-8588, Japan}
\email{hosono@math.gaushuin.ac.jp}
\author[T.-J.~Lee]{Tsung-Ju~Lee}
\address{Tsung-Ju~Lee: Center of Mathematical Sciences and Applications, 20 Garden St., Cambridge, MA 02138, U.S.A.}
\email{tjlee@cmsa.fas.harvard.edu}
\author[B.~H.~Lian]{Bong~H.~Lian}
\address{Bong~H.~Lian, Department of Mathematics, Brandeis University, Waltham MA 02454, U.S.A.}
\email{lian@brandeis.edu}
\author[S.-T.~Yau]{Shing-Tung~Yau}
\address{Shing-Tung~Yau, Department of Mathematics, Harvard University, Cambridge MA 02138, U.S.A.~\& Yau Mathematical Sciences Center, Tsinghua University, Beijing 100084, China}
\email{yau@math.harvard.edu}

\begin{document}
\begin{abstract}
The presented paper is a continuation of the series of papers 
\cites{2018-Hosono-Lian-Takagi-Yau-k3-surfaces-from-configurations-of-six-lines-in-p2-and-mirror-symmetry-i,2019-Hosono-Lian-Yau-k3-surfaces-from-configurations-of-six-lines-in-p2-and-mirror-symmetry-ii}.
In this paper, utilizing Batyrev and Borisov's duality construction on nef-partitions, 
we generalize the recipe in \cites{2018-Hosono-Lian-Takagi-Yau-k3-surfaces-from-configurations-of-six-lines-in-p2-and-mirror-symmetry-i,2019-Hosono-Lian-Yau-k3-surfaces-from-configurations-of-six-lines-in-p2-and-mirror-symmetry-ii} to construct 
a pair of singular double cover Calabi--Yau varieties \((Y,Y^{\vee})\) over toric manifolds and
compute their topological Euler characteristics and Hodge numbers.
In the \(3\)-dimensional cases, we show that \((Y,Y^{\vee})\) forms a topological mirror pair, i.e.,
\(h^{p,q}(Y)=h^{3-p,q}(Y^{\vee})\) for all \(p,q\).
\end{abstract}
\maketitle
\tableofcontents
\section{Introduction}
\subsection{Motivations}
Mirror symmetry from physics has successfully made numerous non-trivial predictions in algebraic geometry
and has been investigated intensively 
in the last decades.
Roughly speaking, a mirror pair is a pair of Calabi--Yau varieties \((M,M^{\vee})\)
such that under certain identification, which is called the \emph{mirror map},
the \(A\)-model correlation function of \(M\) is identified with the
\(B\)-model correlation function of \(M^{\vee}\) and vise versa.

The first mirror pair was written down by Greene and Plesser \cite{GP1990}, the
quintic and the (orbifold) Fermat quintic threefold. Utilizing reflexive polytopes, Batyrev gave a recipe 
for constructing mirror pairs for Calabi--Yau hypersurfaces in Gorenstein toric varieties 
\cite{1994-Batyrev-dual-polyhedra-and-mirror-symmetry-for-calabi-yau-hypersurfaces-in-toric-varieties}.
Soon later Batyrev and Borisov generalized the construction to Calabi--Yau complete
intersections in Gorenstein toric varieties via nef-partitions 
\cite{1996-Batyrev-Borisov-on-calabi-yau-complete-intersections-in-toric-varieties}.

During the last two decades, to test mirror symmetry, many techniques had been developed and
numerous numerical quantities had been calculated explicitly.
The first convincing evidence was the successful prediction 
of the numbers of rational curves
on quintic threefolds in \(\mathbb{P}^{4}\) by Candelas et.~al.~
\cite{1991-Candelas-de-la-Ossa-Green-Parkes-a-pair-of-calabi-yau-manifolds-as-an-exactly-soluable-superconformal-theory}
in a vicinity of the so-called maximal unipotent monodromy point in the moduli.
Hosono, Klemm, Theisen, and Yau calculated the \(B\)-model correlation functions
as well as the mirror maps to
test the mirror symmetry for Calabi--Yau hypersurfaces in
\(4\)-dimensional Gorenstein toric Fano varieties
\cite{1995-Hosono-Klemm-Theisen-Yau-mirror-symmetry-mirror-map-and-applications-to-calabi-yau-hypersurfaces}.
It was observed by Hosono, Lian, and Yau in~
\cite{1996-Hosono-Lian-Yau-gkz-generalized-hypergeometric-systems-in-mirror-symmetry-of-calabi-yau-hypersurfaces} that the
Gr\"{o}bner basis for the toric ideal determines a finite set of differential operators
for the local solutions to the \(A\)-hypergeometric system,
one of the most important tools to study the \(B\)-model correlation
function introduced by Gel'fand, Kapranov and Zelevinskii 
\cite{1989-Gelfand-Kapranov-Zelebinski-hypergeometric-functions-and-toral-manifolds}.
They also proved the existence of rank one points of the \(A\)-hypergemetric system
for the family of Calabi--Yau hypersurfaces in certain toric varieties,
where mirror symmetry is expected 
\cite{1997-Hosono-Lian-Yau-maximal-degeneracy-points-of-gkz-systems}.

The family of \(K3\) surfaces arising from double covers branched along 
six lines in \(\mathbb{P}^{2}\) in general positions
were studied by Matsumoto, Sasaki, and Yoshida 
\cites{1988-Matsumoto-Sasaki-Yoshida-the-period-map-of-a-4-parameter-family-of-k3-surfaces-and-the-aomoto-gelfand-hypergeometric-function-of-type-3-6,1992-Matsumoto-Sasaki-Yoshida-the-monodromy-of-the-period-map-of-a-4-parameter-family-of-k3-surfaces-and-the-hypergeometric-function-of-type-3-6} as a higher dimensional analogue of the Legendre family.
The parameter space \(P(3,6)\) of this \(K3\) family 
admits various compactifications --
\textbf{(a)} a GIT compactification (a.k.a.~the Baily--Borel--Satake compactification)
\cites{1988-Dolgachev-Ortland-point-sets-in-projective-spaces-and-theta-functions,1993-Matsumoto-theta-functions-on-bounded-symmetric-domain-of-type-i-22-and-the-period-map-of-a-4-parameter-family-of-k3-surfaces}
and \textbf{(b)} a toridal compactification constructed by Reuvers
\cite{2006-Reuvers-moduli-spaces-of-configurations}.
However, Hosono, Lian, Takagi, and Yau observed in 
\cite{2018-Hosono-Lian-Takagi-Yau-k3-surfaces-from-configurations-of-six-lines-in-p2-and-mirror-symmetry-i}
that none of these compactifications admits a priori
the so-called \emph{large complex structure limit points} (LCSL points for short hereafter).
In order to study mirror symmetry, they constructed a 
new compactification of \(P(3,6)\) and found 
LCSL points on it by relating it to \textbf{(a)} and \textbf{(b)}.
We briefly explain their idea.
The \(\mathrm{GL}_{3}(\mathbb{C})\)-action on \(\mathbb{P}^{2}\) allows us to
rearrange the hyperplanes to the coordinate axes so that
the \(K3\) family is in fact parameterized by three lines 
in \(\mathbb{P}^{2}\). This procedure is called the \emph{partial gauge fixing}
in \cite{2018-Hosono-Lian-Takagi-Yau-k3-surfaces-from-configurations-of-six-lines-in-p2-and-mirror-symmetry-i}. 
After the partial gauge fixing, it turns out that the period integrals of
our \(K3\) family satisfy certain \(A\)-hypergeometric system
with \(A\in\mathrm{Mat}_{5\times 9}(\mathbb{Z})\),  
a \emph{fractional} exponent \(\beta\in\mathbb{Q}^{5}\).
The matrix \(A\) can be recognized as the integral matrix
associated to certain nef-partition on the base \(\mathbb{P}^{2}\)
and the torus \((\mathbb{C}^{\ast})^{5}\) can be identified with
\(L\otimes\mathbb{C}^{\ast}\), where \(L\) is the lattice relation of \(A\).
Consequently, \(P(3,6)\) admits a toroidal compactification
via the associated secondary fan.
Standard techniques for Calabi--Yau hypersurfaces or complete intersections
in toric varieties are still applicable and results in \cites{1995-Hosono-Klemm-Theisen-Yau-mirror-symmetry-mirror-map-and-applications-to-calabi-yau-hypersurfaces,1996-Hosono-Lian-Yau-gkz-generalized-hypergeometric-systems-in-mirror-symmetry-of-calabi-yau-hypersurfaces,1997-Hosono-Lian-Yau-maximal-degeneracy-points-of-gkz-systems} can be straightened 
into this situation.
Because of this striking similarity with the classical complete intersections, 
we shall call such a double cover a \emph{fractional complete intersection}.
Based on numerical evidences, 
it is conjectured that the mirror of the said \(K3\) family is 
given by certain double covers over a del Pezzo surface of degree \(6\),
which is a blow-up of three torus invariant points on \(\mathbb{P}^{2}\) 
(\cite{2019-Hosono-Lian-Yau-k3-surfaces-from-configurations-of-six-lines-in-p2-and-mirror-symmetry-ii}*{Conjecture 6.3}).
Note that such a del Pezzo surface can be obtained from
Batyrev--Borisov's duality construction for the associated nef-partition
on \(\mathbb{P}^{2}\). 

\subsection{Statements of main results}
The aim of this paper is to study the conjecture 
\cite{2019-Hosono-Lian-Yau-k3-surfaces-from-configurations-of-six-lines-in-p2-and-mirror-symmetry-ii}*{Conjecture 6.3} 
and its further generalization.
Consider a nef-partition \((\Delta,\{\Delta_{i}\}_{i=1}^{r})\)
and its dual nef-partition \((\nabla,\{\nabla_{i}\}_{i=1}^{r})\) in the sense of
Batyrev and Borisov (the precise definitions of nef-partition and the dual nef-partition
will be given in \S\ref{subsection:B-B-duality-construction}).
Let \(\mathbf{P}_{\Delta}\) and \(\mathbf{P}_{\nabla}\)
be the toric varieties defined by \(\Delta\) and \(\nabla\). Let
\(X\to\mathbf{P}_{\Delta}\) and \(X^{\vee}\to \mathbf{P}_{\nabla}\) be
maximal projective crepant partial desingularizations (MPCP desingularizations for short hereafter)
of \(\mathbf{P}_{\Delta}\) and \(\mathbf{P}_{\nabla}\).
The nef-partitions on \(\mathbf{P}_{\Delta}\) and \(\mathbf{P}_{\nabla}\)
determine nef-partitions on \(X\) and \(X^{\vee}\). Let 
\(E_{1},\ldots,E_{r}\) and \(F_{1},\ldots,F_{r}\) be the sum of
toric divisors representing nef-partitions on \(X\) and \(X^{\vee}\), respectively.
In this article, we will assume that
\begin{center}
\label{assumption}
\(X\) and \(X^{\vee}\) are both \emph{smooth}.
\end{center}
Said differently, both
\(\Delta\) and \(\nabla\) admit a \emph{regular triangulation}. By a regular triangulation
of \(\Delta\), we mean that a triangulation of \(\Delta\) such that 
each simplex is regular and contains \(\mathbf{0}\) as a vertex.

Let \(s_{j}\in\mathrm{H}^{0}(X,2E_{j})\) be a smooth section and
\(Y\) be the double cover over \(X\) branched along \(s_{1}\cdots s_{r}\).
Deforming the sections \(s_{j}\) yields a family of Calabi--Yau double covers over 
\(X\), which is parameterized by a suitable open set in 
the product of \(\mathrm{H}^{0}(X,2E_{j})\). 
We now elaborate how to define \emph{a partial gauge fixing} for such a family
(see \S\ref{subsection:partial-gauge-fixings} for details), which turns out to 
be crucial in this paper. 

A \emph{partial gauge fixing} is a decomposition of the section \(s_{j}\) into 
a product of a canonical section of \(E_{j}\) and a smooth section of \(E_{j}\).
In other words, \(s_{j}=s_{j,1}s_{j,2}\) with \(s_{j,k}\in\mathrm{H}^{0}(X,E_{j})\)
such that \(\mathrm{div}(s_{j,1})\equiv E_{j}\) and \(\mathrm{div}(s_{j,2})\) is smooth.
The original double cover family will restrict to a \emph{subfamily} parametrized by
\begin{equation}
V\subset \mathrm{H}^{0}(X,E_{1})\times\cdots\times\mathrm{H}^{0}(X,E_{r}).
\end{equation}
A parallel construction can be applied on the dual side.
Let \(\mathcal{Y}\to V\) and \(\mathcal{Y}^{\vee}\to U\)
be partial gauge fixings for those families.
Let \(Y\) and \(Y^{\vee}\) be the fiber of these families.

We observe that \(Y\) and \(Y^{\vee}\) form a topological mirror pair.
\begin{theorem}[={\bf Theorem}~\ref{theorem:topological-mirror-dualty}]
We have \(\chi_{\mathrm{top}}(Y)=(-1)^{n}\chi_{\mathrm{top}}(Y^{\vee})\), where 
\(n=\dim Y\) and \(\chi_{\mathrm{top}}(-)\) denotes
the topological Euler characteristic.
\end{theorem}
Since \(Y\) and \(Y^{\vee}\) are \emph{orbifolds}, the Hodge numbers \(h^{p,q}(Y)\) are 
well-defined. Moreover, by construction, \(X\setminus B\) is affine, where \(B\)
is the branch locus of the cover \(Y\to X\). It follows that
\(h^{p,q}(Y)=h^{p,q}(X)\) for all \(p,q\) with \(p+q\ne n\).
In particular, when \(n=3\), we can prove
\begin{theorem}[={\bf Theorem}~\ref{theorem:hodge-numbers-for-cy-3-folds}]
We have \(h^{p,q}(Y)=h^{3-p,q}(Y^{\vee})\) for all \(p,q\).
\end{theorem}
The calculation of the Euler characteristics boils down to
computation of intersection numbers on toric varieties, which
turns out to be a consequence of a combinatorial formula
by Danilov and Khovanskii \cite{1986-Danilov-Khovanskii-newton-polyhedra-and-an-algorithm-for-calculating-hodge-deligne-numbers}. 

Based on these results, we propose the following conjecture, 
which can be served as a generation of 
\cite{2019-Hosono-Lian-Yau-k3-surfaces-from-configurations-of-six-lines-in-p2-and-mirror-symmetry-ii}*{Conjecture 6.3}.
\begin{conjecture}
\(Y\) is mirror to \(Y^{\vee}\).
\end{conjecture}
We shall emphasize that none of \(Y\) and \(Y^{\vee}\) is smooth. 
The conjecture is served as an extension of the classical mirror correspondence
to \emph{singular Calabi--Yau varieties}.

\begin{remark}
The quantum test, i.e., the correspondence between
enumerative geometry (more precisely, the
Chen--Ruan orbifold Gromov--Witten invariants) and complex geometry
(deformation of complex structures), for the conjecture will be treated in
our forthcoming paper. 
\end{remark}

We work over \(\mathbb{C}\), the field of complex numbers.

\subsection{Acknowledgment}
We thank Center of Mathematical Sciences and Applications
at Harvard for hospitality while working on this project. 
S.~Hosono is supported in part by Grant-in Aid Scientific Research (C 16K05105, S 17H06127, A 18H03668).
B.~Lian and S.-T.~Yau are supported by the Simons Collaboration Grant on Homological Mirror Symmetry and Applications 2015-2022.

\section{Preliminaries}
\label{section:preliminaries}
\subsection{Cyclic covers}
\label{section:cyclic-covers-general-setup}
In this paragraph, let \(X\) be a smooth projective variety, \(L\) be a line bundle over \(X\)
and \(\mathscr{L}\) be the sheaf of sections of \(L\).
For \(s\in\mathrm{H}^{0}(X,L^{r})=
\mathrm{Hom}_{\mathscr{O}_{X}}(\mathscr{O}_{X},\mathscr{L}^{r})\), the dual
\(s^\vee:{\mathscr{L}^{-r}}\to\mathscr{O}_X\) determines an 
\(\mathscr{O}_X\)-algebra structure on 
\(\mathscr{A}_s':=\bigoplus_{i=0}^{r-1}\mathscr{L}^{-i}\).
In fact, we have an identification
\begin{equation}
\mathscr{A}_{s}' = \bigoplus_{i=0}^{\infty} \mathscr{L}^{-i}\big\slash \mathscr{I},
\end{equation}
where \(\mathscr{I}\) is the sheaf of \((\bigoplus_{i=0}^{\infty} \mathscr{L}^{-i})\)-module generated by 
\begin{equation}
\{s^{\vee}(\ell)-\ell:\ell~\mbox{is a local section of}~\mathscr{L}^{-r}\}.
\end{equation}
Note that the multiplication on \(\mathscr{A}_{s}'\) is
given by the usual multiplication on sections 
\(\mathscr{L}^{-i}\times \mathscr{L}^{-j} \to\mathscr{L}^{-i-j}\) and further
composed with \(s^{\vee}\) if \(i+j\ge r\).
Let \(Y_s':=\mathit{Spec}_{\mathscr{O}_{X}}(\mathscr{A}_s')\)
and \(Y_{s}\to Y_{s}'\) be the normalization.
We denote by \(D_{s}\) the scheme-theoretic zero of \(s\).
\begin{definition}
\label{definition:cyclic-covers}
The scheme \(Y_{s}\) is called the 
\emph{\(r\)-fold cyclic cover over \(X\) branched over \(D_{s}\)} or
simply \emph{the \(r\)-fold cover} if the context is clear.
\end{definition}

We are mainly interested in the situation
that \(\mathrm{codim}_X\mathrm{Sing}(D_s)\ge 2\),
which implies that \(Y_{s}'\) is already normal and 
consequently \(Y_{s}=Y_{s}'\). 
We denote by \(\omega_{X}\) and \(\omega_{Y}\) the dualizing sheaf of \(X\)
and \(Y\). We can summarize these results 
in the next proposition.

\begin{proposition}
\label{proposition:CY-condition}
\(Y_{s}'\) is Cohen--Macaulay. Furthermore,
if \(\mathrm{codim}_X\mathrm{Sing}(D_s)\ge 2\), then \(Y_{s}'\)
is a normal variety.
\(\omega_{Y}\simeq \mathscr{O}_{Y}\) if and only if
\(\omega_X\otimes\mathscr{L}^{r-1}\simeq\mathscr{O}_X\). 
\end{proposition}
\begin{proof}
See Proposition \ref{proposition:calabi-yau-condition-singular} for the proof.
\end{proof}

\begin{instance} 
Let \(X=\mathbb{P}^{n}\) and \(\mathscr{L}=\mathscr{O}_{X}(d)\).
We list some \(r\)-fold cyclic covers over \(X\) which satisfy 
\(\omega_{Y}\simeq \mathscr{O}_{Y}\). In this case, the criterion in
Proposition \ref{proposition:CY-condition} boils down 
to the numerical constraint \(n+1=d(r-1)\).
\begin{itemize}
  \item \(n=1\).
    \begin{itemize}
      \item[(1a)] \(d=r=2\). The cyclic cover \(Y\) is an elliptic curve and 
      the attached family is known as the Legendre family. 
      The general fiber (branched over four distinct points) has non-zero \(j\)-invariant.
      \item[(1b)] \(d=1\) and \(r=3\). The cyclic cover \(Y\) is also an elliptic curve, 
      whose \(j\)-invariant is zero.
    \end{itemize}
  \item \(n=2\).
    \begin{itemize}
      \item[(2a)] \(d=3\) and \(r=2\). 
      \item[(2b)] \(d=1\) and \(r=4\). 
    \end{itemize}
  \item \(n=3\).
    \begin{itemize}
      \item[(3a)] \(d=4\) and \(r=2\). 
      \item[(3b)] \(d=2\) and \(r=3\).
      \item[(3c)] \(d=1\) and \(r=5\).
    \end{itemize}
\end{itemize}
\end{instance}

We can also compute the Euler characteristic for the cyclic covers.
Let us recall that for an \(n\)-dimensional complex analytic variety \(W\), the Euler characteristic 
is defined to be
\begin{equation*}
\chi(W) := \sum_{k=0}^{2n} (-1)^{k}\dim \mathrm{H}^{k}(W)
=\sum_{k=0}^{2n} (-1)^{k}\dim \mathrm{H}_{c}^{k}(W).
\end{equation*}
If \(U\to W\) is a finite \'{e}tale cover of degree \(r\), 
then we have \(\chi(U)=r\cdot\chi(W)\).

Let \(\pi:Y\to X\) be an \(r\)-fold cyclic cover and \(D\) be the ramification locus. 
Then \(Y\setminus \pi^{-1}(D)\to X\setminus D\) is a finite \'{e}tale cover of degree \(r\).
We then have 
\begin{align}
\label{equation:euler-charactersitic-branched-covers}
\begin{split}
\chi(Y) &= \chi(\pi^{-1}(D)) + \chi(Y\setminus \pi^{-1}(D))\\ 
&= \chi(D) + r\cdot\chi(X\setminus D)\\
&= \chi(D) + r(\chi(X)-\chi(D)).
\end{split}
\end{align}

\subsection{Toric varieties and Batyrev--Borisov's duality construction}
\label{subsection:B-B-duality-construction}
To elaborate the singular mirror duality in this paper, 
we review the construction of classical mirror duality pair of Calabi--Yau complete
intersections in toric varieties introduced by
Batyrev and Borisov
\cite{1996-Batyrev-Borisov-on-calabi-yau-complete-intersections-in-toric-varieties}.
Let us begin with the following data.
\begin{itemize}
  \item Let \(N=\mathbb{Z}^n\) be a lattice of rank \(n\) 
  and \(M:=\mathrm{Hom}_{\mathbb{Z}}(N,\mathbb{Z})\) be the dual lattice. 
  We denote by \(N_\mathbb{R}\) and \(M_\mathbb{R}\) the tensor products
  \(N\otimes_{\mathbb{Z}}\mathbb{R}\) and \(M\otimes_{\mathbb{Z}}\mathbb{R}\).
  \item For a complete fan \(\Sigma\) in \(N_{\mathbb{R}}\), we denote by \(\Sigma(k)\) 
  the set of all \(k\)-dimensional cones in \(\Sigma\). 
  For convenience, we write \(\Sigma(1)=\{\rho_1,\ldots,\rho_p\}\).
  The same notation \(\rho_i\) is used to denote 
  the primitive generator of the corresponding 1-cone. 
  The support of \(\Sigma\) is denoted by \(|\Sigma|\).
  \item The toric variety defined by \(\Sigma\) is 
  denoted by \(X_\Sigma\) or simply by \(X\) if the context is clear. 
  Let \(T=(\mathbb{C}^{\ast})^n\) be its maximal torus.
  Each \(\rho\in\Sigma(1)\) determines a Weil divisor \(D_{\rho}\) on \(X\).
  \item Let \(D=\sum_{\rho} a_{\rho} D_{\rho}\) be a torus invariant divisor. 
  The divisor polytope \(\Delta_{D}\) is defined by
  \begin{equation*}
    \Delta_D:=\{m\in M_{\mathbb{R}}:\langle m,\rho\rangle\ge -a_{\rho},
    ~\forall \rho\in\Sigma(1)\}.
  \end{equation*}
  \item A polytope in \(M_{\mathbb{R}}\) is called lattice polytope
  if its vertices belong to \(M\). For a lattice polytope \(\Delta\)
  in \(M_{\mathbb{R}}\), we denote by \(\Sigma_{\Delta}\) the normal fan of 
  \(\Delta\). The toric variety determined by \(\Delta\) is denoted by \(\mathbf{P}_{\Delta}\),
  i.e., \(\mathbf{P}_{\Delta}=X_{\Sigma_{\Delta}}\).
  \item A reflexive polytope \(\Delta\subset M_{\mathbb{R}}\) is a lattice polytope 
  containing the origin \(0\in M_{\mathbb{R}}\) in its interior and such that the polar dual 
  \(\Delta^{\vee}\) is again a lattice polytope. If
  \(\Delta\) is a reflexive polytope, then \(\Delta^{\vee}\) is also a lattice
  polytope and satisfies \((\Delta^{\vee})^{\vee}=\Delta\). The normal fan of \(\Delta\)
  is the face fan of \(\Delta^{\vee}\) and vice versa.

  \end{itemize}
Let \(I_{1},\ldots,I_{r}\) be a nef-partition on \(\mathbf{P}_{\Delta}\),
that is, \(\Sigma_{\Delta}(1)=\sqcup_{s=1}^r I_{s}\) and \(E_{s}:=\sum_{\rho\in I_{s}} D_{\rho}\)
is numerical effective for each \(s\). This gives rise to a Minkowski sum decomposition
\(\Delta=\Delta_{1}+\cdots+\Delta_{r}\), where \(\Delta_{i}=\Delta_{E_{i}}\)
is the section polytope of \(E_{i}\).
The Batyrev--Borisov duality construction
goes in the following way.

Let \(\nabla_{k}\) be the convex hull of \(\{\mathbf{0}\}\cup I_{k}\) and
\(\nabla=\nabla_1+\ldots+\nabla_{r}\) be their Minkowski sum.
It turns out that \(\nabla\) is a reflexive polytope in \(N_{\mathbb{R}}\)
whose polar polytope is given by \(\nabla^{\vee}=\mathrm{Conv}(\Delta_{1},\ldots,\Delta_{r})\)
and \(\nabla_1+\ldots+\nabla_{r}\) corresponds to a nef-partition on \(\mathbf{P}_{\nabla}\),
called the \emph{dual nef-partition}.
The corresponding nef toric divisors are denoted by \(F_{1},\ldots,F_{r}\).
Then the section polytope of \(F_{j}\) is \(\nabla_{j}\).

Let \(X\to \mathbf{P}_{\Delta}\) and \(X^{\vee}\to \mathbf{P}_{\nabla}\)
be MPCP desingularizations
for \(\mathbf{P}_{\Delta}\) and \(\mathbf{P}_{\nabla}\).
Via pullback, the nef-partitions on \(\mathbf{P}_{\Delta}\) and \(\mathbf{P}_{\nabla}\)
determine nef-partitions on \(X\) and \(X^{\vee}\) and they determine 
the families of Calabi--Yau complete intersections inside \(X\) and \(X^{\vee}\) respectively.

Recall that the section polytopes \(\Delta_{i}\) and \(\nabla_{j}\)
correspond to \(E_{i}\) on \(\mathbf{P}_{\Delta}\) and 
\(F_{j}\) on \(\mathbf{P}_{\nabla}\), respectively.
To save the notation, the corresponding nef-partitions and toric divisors 
on \(X\) and \(X^{\vee}\) will be still denoted by \(\Delta_{i}\), \(\nabla_{j}\) and
\(E_{i}\), \(F_{j}\) respectively.

There is another point of view which is useful for us.
Given a nef-partition on \(X\) as above, 
corresponding to \(\Delta=\Delta_{1}+\cdots+\Delta_{r}\), 
one constructs a cone in \(\mathbb{R}^{r}
\times M_{\mathbb{R}}\) by 
\begin{equation*}
\sigma_{\Delta}:=\left\{\left(\lambda_{1},\ldots,\lambda_{r},\sum_{i=1}^{r} 
\lambda_{i}w_{i}\right): w_{i}\in \Delta_{i}~\mbox{and}~\lambda_{i}\ge 0\right\}.
\end{equation*}
Then the dual cone \(\sigma_{\Delta}^{\vee}\subset \mathbb{R}^{r}\times N_{\mathbb{R}}\) 
can be identified with the cone \(\sigma_{\nabla}\subset \mathbb{R}^{r}\times N_{\mathbb{R}}\) 
constructed from the dual nef-partition \(\nabla_{1}+\cdots+\nabla_{r}\). 
\(\sigma_{\Delta}\) and \(\sigma_{\nabla}\) arising in this way give
a pair \((\sigma_{\Delta},\sigma_{\nabla})\) of the so-called reflexive
Gorenstein cones with index \(r\). 
See \cite{2008-Batyrev-Nill-combinatorial-aspects-of-mirror-symmetry} for further discussions.

The following proposition may be known to experts. 
\begin{proposition}
\label{proposition:euler-characteristic-lemma}
Assume that \(X\) and \(X^{\vee}\) are both smooth.
Let \(\{\mathrm{e}_{i}\}_{i=1}^{r}\) be the standard basis of \(\mathbb{R}^{r}\).
We denote by \(S\) the convex hull of \(\mathbf{0}\) 
and \(\mathrm{e}_{i}\times(\Delta_{i}\cap M)\), \(i=1,\ldots,r\), 
in \(\mathbb{R}^{r}\times M_{\mathbb{R}}\).
Then the normalized volume of \(S\) in \(\mathbb{R}^{r}\times M_{\mathbb{R}}\) is
equal to the normalized volume of \(\nabla^{\vee}=\mathrm{Conv}(\Delta_{1},\ldots,\Delta_{r})\) 
in \(M_{\mathbb{R}}\).
\end{proposition}
\begin{proof}
Let \(W\) be the total space of the vector bundle 
\(\mathscr{O}(F_{1})\oplus\cdots\oplus\mathscr{O}(F_{r})\) over
\(X^{\vee}\). Since \(X^{\vee}\) is assumed to be smooth, \(W\) is a smooth toric variety 
having the same Euler characteristic with \(X^{\vee}\).
The normalized volume of \(S\) is equal to the number of maximal
cones in the toric variety \(W\) and therefore
it is equal to the Euler characteristic of \(X^{\vee}\), that is,
the normalized volume of \(\nabla^{\vee}\).
\end{proof}

Let \(Z_{1},\ldots,Z_{k}\) be nef torus 
invariant divisors on \(X\) and \(\Delta_{Z_{i}}\)
be the section polytope of \(Z_{i}\).
Let \(\Delta_{Z_{1}}\star\cdots\star\Delta_{Z_{k}}\) be the Cayley
polytope of \(\Delta_{Z_{i}}\), i.e.,
the convex hull of the polyhedra 
\(\mathrm{e}_{1}\times\Delta_{Z_{1}},\ldots,\mathrm{e}_{k}\times\Delta_{Z_{k}}\)
in the space \(\mathbb{R}^{k}\times M_{\mathbb{R}}\).
Similarly for each nonempty subset \(J\subset\{1,\ldots,k\}\), we define
\(\Delta^{\star J}:=\star_{j\in J} \Delta_{Z_{j}}\subset 
\mathbb{R}^{|J|}\times M_{\mathbb{R}}\).
Let \(\Lambda\) and \(\Lambda_{J}\) be the pyramids with vertex \(\mathbf{0}\) and 
base \(\Delta_{1}\star\cdots\star\Delta_{r}\) and \(\Delta^{\star J}\) 
in \(\mathbb{R}^{r}\times M_{\mathbb{R}}\) and \(\mathbb{R}^{|J|}\times M_{\mathbb{R}}\) respectively.
Now we can state a result due to Danilov and Khovanskii.
\begin{theorem}[cf.~\cite{1986-Danilov-Khovanskii-newton-polyhedra-and-an-algorithm-for-calculating-hodge-deligne-numbers}*{\S6}]
\label{theorem:DK-euler-characteristic}
For general \(D_{i}\) in the linear system \(|Z_{i}|\), we have
\begin{equation*}
\chi(D_{1}\cap\cdots\cap D_{k}\cap T) =
-\sum_{J} (-1)^{n+|J|-1}\mathrm{vol}_{n+|J|}(\Lambda_{J}),
\end{equation*}
where the summation runs over all nonempty subsets \(J\subset\{1,\ldots,k\}\)
and \(\mathrm{vol}_{k}\) is the normalized volume in \(k\)-dimensional spaces.
\end{theorem}

\begin{instance}[Double covers over \(\mathbb{P}^{2}\) branched along six lines]
\label{instance:double-cover-over-p-2}
Let \(X=\mathbb{P}^{2}\) and \(\Delta=\mathrm{Conv}\{(2,-1),(-1,2),(-1,-1)\}\) 
be the section polytope of \(-K_{X}\). 
We denote by \(\rho_{1}\), \(\rho_{2}\), and \(\rho_{3}\)
the primitive vectors \((1,0)\), \((0,1)\), and \((-1,-1)\) respectively 
generating the \(1\)-cones of the normal fan of \(\Delta\), i.e.,
the standard \(\mathbb{P}^{2}\) fan. 
Then the divisors \(E_{i}:=D_{\rho_{i}}\), \(i=1,2,3\),
define a nef-partition on \(X=\mathbf{P}_{\Delta}\). Correspondingly 
we have the decomposition \(\Delta=\Delta_{1}+\Delta_{2}+\Delta_{3}\).
Batyrev--Borisov duality applies to \(\Delta\) and \(\nabla\) such
that \(\nabla^{\vee}=\mathrm{Conv}(\Delta_{1},\Delta_{2},\Delta_{3})\).
Namely if we define \(\nabla_{i}:=\mathrm{Conv}(\mathbf{0},\rho_{i})\),
then we obtain the decomposition \(\nabla=\nabla_{1}+\nabla_{2}+\nabla_{3}\)
which corresponds to the dial nef-partition \(F_{1}+F_{2}+F_{3}\)
on \(X^{\vee}=\mathbf{P}_{\nabla}\).
Having \(\nabla=\mathrm{Conv}\{(\pm 1,0),(0,\pm 1),(1,1),(-1,-1)\}\), 
or from the face fan of \(\nabla^{\vee}\),
we determine the normal fan of \(\nabla\),
which is described by the following primitive generators of \(1\)-dimensional cones
\begin{align*}
\nu_{1}&=(-1,1),~\nu_{2}=(-1,0),~\nu_{3}=(0,-1),\\
\nu_{4}&=(1,-1),~\nu_{5}=(1,0),~\mbox{and}~\nu_{6}=(0,1).
\end{align*} 
From these data, we see that \(\mathbf{P}_{\nabla}\) is isomorphic to \(\mathbb{P}^{2}\)
blown up at three points. Also the dual nef-partition is given by 
\begin{equation*}
F_{k} = D_{\nu_{2k-1}}+D_{\nu_{2k}},~k=1,2,3.
\end{equation*}

\end{instance}

\section{Mirror symmetry for singular Calabi--Yau double covers over toric manifolds}

The mirror duality between singular \(K3\) surfaces (see Example \ref{instance:family-of-singular-k3-sufraces} below) was discovered
in \cites{2018-Hosono-Lian-Takagi-Yau-k3-surfaces-from-configurations-of-six-lines-in-p2-and-mirror-symmetry-i,2019-Hosono-Lian-Yau-k3-surfaces-from-configurations-of-six-lines-in-p2-and-mirror-symmetry-ii}.
In this paragraph, we put the mirror duality into a more general framework:
we formulate the mirror duality for the pair of singular
Calabi--Yau varieties, which are double covers over certain pair of dual toric manifolds.

Let us keep the notation in \S\ref{subsection:B-B-duality-construction}.
Starting with a reflexive polytope \(\Delta\)  in \(M_{\mathbb{R}}\) 
and a decomposition \(\Delta_{1}+\cdots+\Delta_{r}\)
representing a nef-partition \(E_{1}+\cdots+E_{r}\) of \(-K_{\mathbf{P}_{\Delta}}\), 
we have the corresponding dual polytope \(\nabla\) in \(N_{\mathbb{R}}\) and 
the dual decomposition \(\nabla_{1}+\cdots+\nabla_{r}\)
representing the dual nef-partition \(F_{1}+\cdots+F_{r}\) of \(-K_{\mathbf{P}_{\nabla}}\).
Let \(X\) and \(X^{\vee}\) be the MPCP desingularization 
of \(\mathbf{P}_{\Delta}\) and \(\mathbf{P}_{\nabla}\) respectively.
Hereafter, we will simply call the decompoition \(\Delta=\Delta_{1}+\cdots+\Delta_{r}\)
a nef-partition on \(X\) for short with understanding the nef-partition \(E_{1}+\cdots +E_{r}\).
Likewise for the decomposition \(\nabla=\nabla_{1}+\cdots+\nabla_{r}\).
Also, unless otherwise stated, we assume that 

\begin{center}
{\it \(X\) and \(X^{\vee}\) are both smooth}.
\end{center}
Equivalently, we assume that both \(\Delta\) and \(\nabla\) admit uni-modular triangulations.
From the duality, we have
\begin{equation*}
\mathrm{H}^{0}(X^{\vee},F_{i})\simeq \bigoplus_{\rho\in\nabla_{i}\cap N}\mathbb{C}\cdot t^{\rho}~\mbox{and}~
\mathrm{H}^{0}(X,E_{i})\simeq \bigoplus_{m\in\Delta_{i}\cap M}\mathbb{C}\cdot t^{m}.
\end{equation*}
Here we use the same notation \(t=(t_{1},\ldots,t_{n})\) to 
denote the coordinates on the maximal torus of \(X\) and \(X^{\vee}\).

From Proposition \ref{proposition:CY-condition}, a double cover \(Y\) has trivial 
canonical bundle if and only if \(\mathscr{L}\simeq \omega_{X}^{-1}\). 
The branch locus of \(Y\to X\) is linearly equivalent to \(-2K_{X}\).

\begin{definition}
Given a decomposition \(\Delta=\Delta_{1}+\cdots+\Delta_{r}\) representing
a nef-partition \(E_{1}+\cdots+E_{r}\) on \(X\),
the double covers \emph{branched along the nef-partition} over \(X\) is
the double cover \(Y\to X\) constructed from the section \(s=s_{1}\cdots s_{r}\) with
\begin{equation*}
(s_{1},\ldots,s_{r})\in \mathrm{H}^{0}(X,2E_{1})\times\cdots\times\mathrm{H}^{0}(X,2E_{r}),
\end{equation*}
where \(E_{i}\) is the corresponding toric divisor to \(\Delta_{i}\).
\end{definition}

\begin{instance}[Families of singular \(K3\) surfaces]
\label{instance:family-of-singular-k3-sufraces}
Let us retain the notation in Example \ref{instance:double-cover-over-p-2}.
Let \(Y\to X:=\mathbb{P}^{2}\) be the double cover branched along 
six lines in general positions. \(Y\) is a singular \(K3\) surface with \(15\)
\(A_{1}\)-singularities.

The indicial ring for
the Picard--Fuchs equation of this family was calculated 
in \cite{2019-Hosono-Lian-Yau-k3-surfaces-from-configurations-of-six-lines-in-p2-and-mirror-symmetry-ii}*{Proposition 4.4}.
It was proved that the intersection pairing \(\langle \theta_{i},\theta_{j}\rangle\)
is identical to the intersection matrix of the divisors \(\tilde{L}_{i}\) in \(Y^{\vee}\), \(i=1,\ldots,4\).
Here \(Y^{\vee}\) is a double cover over \(X^{\vee}\) and \(\tilde{L}_{i}\) is the pullback of
\(L_{i}=F_{i}\) for \(i=1,2,3\) and \(L_{4}=H\) on \(X^{\vee}\) 
(the pullback of the hyperplane class on \(X^{\vee}\to \mathbb{P}^{2}\)).
For notation and details, 
see \cite{2019-Hosono-Lian-Yau-k3-surfaces-from-configurations-of-six-lines-in-p2-and-mirror-symmetry-ii}*{\S 4 and \S 6}.
\end{instance}

In order to generalize the duality construction to double covers over toric varieties, we 
need the concept of ``partial gauge fixings''.

\subsection{Partial gauge fixings}
\label{subsection:partial-gauge-fixings}
In the \(K3\) example, the gauge fixed family over \(\mathbb{P}^{2}\)
is the subfamily when the ``half'' of the branched divisors are fixed 
to be the toric divisors. 
Inspired by this, we are led to consider the case when
\(s_{i}\in\mathrm{H}^{0}(X,2E_{i})\) is 
of the form \(s_{i}=s_{i,1}s_{i,2}\) with \(s_{i,1},s_{i,2}\in \mathrm{H}^{0}(X,E_{i})\).
We further assume that \(s_{i,1}\) is the section corresponding to the lattice point
\(\mathbf{0}\in\Delta_{i}\cap M\), i.e., the scheme-theoretic zero of \(s_{i,1}\)
is \(E_{i}\), and that the scheme-theoretic zero of \(s_{i,2}\) is non-singular. 
In this manner, we obtain a subfamily of double covers branched along
the nef-partition over \(X\) parameterized by an open subset
\begin{equation*}
V\subset \mathrm{H}^{0}(X,E_{1})\times\cdots\times\mathrm{H}^{0}(X,E_{r}).
\end{equation*}
\begin{definition}
Given a decomposition \(\Delta=\Delta_{1}+\cdots+\Delta_{r}\) representing
a nef-partition \(E_{1}+\cdots+E_{r}\) on \(X\),
the subfamily \(\mathcal{Y}\to V\) constructed above is called the \emph{gauge fixed double cover branched along
the nef-partition over \(X\)} or simply the \emph{gauge fixed double cover} if 
no confuse occurs.
\end{definition}

Given a decomposition \(\Delta=\Delta_{1}+\cdots+\Delta_{r}\) representing
a nef-partition \(E_{1}+\cdots+E_{r}\) on \(X\) as above,
we denote by \(\mathcal{Y}\to V\) the gauge fixed double cover family. 
A parallel construction is applied for the dual decomposition \(\nabla=\nabla_{1}+\cdots
+\nabla_{r}\) representing the dual nef-partition \(F_{1}+\cdots+F_{r}\) 
over \(X^{\vee}\) and this yields
another family \(\mathcal{Y}^{\vee}\to U\),
where \(U\) is an open subset in 
\begin{equation*}
\mathrm{H}^{0}(X^{\vee},F_{1})\times\cdots\times\mathrm{H}^{0}(X^{\vee},F_{r}).
\end{equation*} 
This construction generalizes our previous example on double covers over \(\mathbb{P}^{2}\). 

\begin{instance}[Families of singular \(K3\) surfaces continued]
\label{instance:families-of-singular-k3-surfaces-continued}
Let \(Y^{\vee}\) be the gauged fixed double cover 
branched along the nef-partition \(F_{1}+F_{2}+F_{3}\) over
\(X^{\vee}\). Let us write down the period integral for the family 
\(\mathcal{Y}^{\vee}\to U\).

Let \(w_{1},\ldots,w_{6}\) be the homogeneous coordinates corresponding to 
divisors \(D_{\nu_{1}},\ldots,D_{\nu_{6}}\) for \(X^{\vee}\).
Let \( t_{1},t_{2} \) be the coordinates
on the maximal torus of \(X^{\vee}\). These are related by \(t_{i}=\prod_{j} w_{j}^{\nu_{j,i}}\) (\(i=1,2\)),
which gives
\begin{equation}
t_{1}=w_{1}^{-1}w_{2}^{-1}w_{4}w_{5},~t_{2}=w_{1}w_{3}^{-1}w_{4}^{-1}w_{6}.
\end{equation}
In terms of homogeneous coordinates, we have
\(s_{1,1}=w_{1}w_{2}\), \(s_{2,1}=w_{3}w_{4}\), and \(s_{3,1}=w_{5}w_{6}\)
for each \(\mathbf{0}\in\Delta_{i}\) representing \(\mathrm{H}^{0}(X^{\vee},F_{i})\).
For the other half of sections \(s_{i,2}\in\mathrm{H}^{0}(X^{\vee},F_{i})\),
we write them with parameters \((a_{1},b_{1},a_{2},b_{2},a_{3},b_{3})\) as follows:
\begin{align*}
s_{1,2}&=a_{1}w_{1}w_{2}+b_{1}w_{4}w_{5}\\
s_{2,2}&=a_{2}w_{3}w_{4}+b_{2}w_{1}w_{6}\\
s_{3,2}&=a_{3}w_{5}w_{6}+b_{3}w_{2}w_{3}.
\end{align*}
Then we can write the period integral as a function on \(U\)
\begin{align*}
\begin{split}
\label{equation:period-integral-mirror-p2}
\int \frac{\Omega_{X^{\vee}}}{\sqrt{s_{1,1}s_{2,1}s_{3,1}s_{1,2}s_{2,2}s_{3,2}}}
=\int \frac{\Omega_{X^{\vee}}}{w_{1}w_{2}w_{3}w_{4}w_{5}w_{6}}\frac{1}{\sqrt{h_{1}h_{2}h_{3}}}
=\int \frac{\mathrm{d}t_{1}\wedge\mathrm{d}t_{2}}{t_{1}t_{2}}\frac{1}{\sqrt{h_{1}h_{2}h_{3}}},
\end{split}
\end{align*}
where 
\begin{align*}
h_{1}&=w_{1}^{-1}w_{2}^{-1}s_{1,2}=a_{1}+b_{1}t_{1},\\
h_{2}&=w_{3}^{-1}w_{4}^{-1}s_{2,2}=a_{2}+b_{2}t_{2},\\
h_{3}&=w_{5}^{-1}w_{6}^{-1}s_{3,2}=a_{3}+b_{3}t_{1}^{-1}t_{2}^{-1}
\end{align*}
and \(\Omega_{X^{\vee}}\) is a generator in \(\mathrm{H}^{0}(X^{\vee},\Omega^{2}_{X^{\vee}}(-K_{X^{\vee}}))\).
It is straightforward to prove that the period integrals are governed 
by the GKZ \(A\)-hypergeometric equations with
\begin{equation*}
A=\begin{bmatrix}
  1 & 1 & 0 & 0 & 0 &  0 \\
  0 & 0 & 1 & 1 & 0 &  0 \\
  0 & 0 & 0 & 0 & 1 &  1 \\
  0 & 1 & 0 & 0 & 0 & -1 \\
  0 & 0 & 0 & 1 & 0 & -1 \\
\end{bmatrix},~
\beta=
\begin{bmatrix}
-1/2\\
-1/2\\
-1/2\\
0\\
0\\
\end{bmatrix}.
\end{equation*}
\end{instance}

\subsection{Topological mirror duality}
Let \(Y\) and \(Y^{\vee}\) be the general fiber in \(\mathcal{Y}\to V\) and 
\(\mathcal{Y}^{\vee}\to U\). 
Note that by construction, \(Y\) and \(Y^{\vee}\)
have trivial canonical bundles with at worst quotient singularities. 

\begin{theorem}
\label{theorem:hodge-number-p-ne-q}
The Hodge numbers
\(h^{p,q}(Y)\) and \(h^{p,q}(Y^{\vee})\) are well-defined and they
are equal to \(h^{p,q}(X)\) and \(h^{p,q}(X^{\vee})\) respectively for \(p+q\ne n\) 
\end{theorem}
\begin{proof}
We defer a proof in Appendix~\ref{section:generalities-on-cyclic-covers} 
(cf.~Proposition \ref{proposition:decomposition-pushforward-log-differential})
where we also provided some generalities about cyclic covers.
\end{proof}

First of all, under our hypothesis on \(X\) and \(X^{\vee}\), we have
\begin{theorem}
\label{theorem:topological-mirror-dualty}
\(\chi_{\mathrm{top}}(Y) = (-1)^{n}\chi_{\mathrm{top}}(Y^{\vee})\).
\end{theorem}
\begin{proof}
For simplicity, we put \(\chi\equiv\chi_{\mathrm{top}}\) in the proof.
We denote by \(E_{i,1}\) and \(E_{i,2}\) the scheme-theoretic zero of \(s_{i,1}\)
and \(s_{i,2}\) respectively. Note that \(E_{i,1}=E_{i}\) and \(\cup_{i=1}^{r} E_{i,1}\)
equals to the union of all toric divisors on \(X\).
Under our gauge fixing, 
the Euler characteristic of the branch locus \(D\) for \(Y\to X\) is 
\begin{align*}
\chi(D) &= \chi(\cup_{i=1}^{r} E_{i,1}) + \chi(T\cap (E_{1,2}\cup\cdots\cup E_{r,2}))\\
        &= \chi(X) + \chi(T\cap (E_{1,2}\cup\cdots\cup E_{r,2})).
\end{align*}
Therefore, from \eqref{equation:euler-charactersitic-branched-covers}, 
we can compute
\begin{align*}
\chi(Y) &= 2\chi(X) - \chi(D)\\
        &=\chi(X) - \chi(T\cap (E_{1,2}\cup\cdots\cup E_{r,2})).
\end{align*}
By inclusion-exclusion principle, Theorem \ref{theorem:DK-euler-characteristic}, 
and Proposition \ref{proposition:euler-characteristic-lemma}, 
\begin{align}
\begin{split}
- \chi(T\cap (E_{1,2}\cup\cdots\cup E_{r,2})) &= 
(-1)^{r-1}\cdot(-1)^{n+r-1} \mathrm{vol}_{n+r}(\Lambda)\\ &
= (-1)^{n} \chi(X^{\vee}).
\end{split}
\end{align}
Hence we have
\begin{align*}
\chi(Y)&=\chi(X)+(-1)^{n}\chi(X^{\vee})\\
&=(-1)^{n}(\chi(X^{\vee})+(-1)^{n}\chi(X))\\
&=(-1)^{n}\chi(Y^{\vee}).
\end{align*}
\end{proof}

In the case of Calabi--Yau threefolds, 
having Euler characteristic and all the Hodge numbers \(h^{p,q}\) with \(p+q\ne 3\) in hand, 
we can completely determine the Hodge diamond.
In fact, we have
\begin{theorem}
\label{theorem:hodge-numbers-for-cy-3-folds}
When \(n=3\), we have \(h^{p,q}(Y)=h^{3-p,q}(Y^{\vee})\) for all \(p,q\).
\end{theorem}
\begin{proof}
A priori we have \(h^{p,q}(Y)=h^{p,q}(X)=0\) for all \(p+q\ne 3\) and \(p\ne q\) since
\(X\) is a toric manifold. Note that \(\chi(X) = 2(1+h^{1,1}(X))\) by Serre duality and
\(\chi(Y) = \chi(X)-\chi(X^{\vee})=2(h^{1,1}(X)-h^{1,1}(X^{\vee}))\). 
Therefore, we have 
\begin{equation*}
h^{2,1}(Y) = h^{1,1}(Y) - \frac{\chi(Y)}{2} = h^{1,1}(X^{\vee}) = h^{1,1} (Y^{\vee}),
\end{equation*}
where the last equality follows from Proposition \ref{proposition:decomposition-pushforward-log-differential}.
\end{proof}

Based on the numerical results, we propose that
\begin{conjecture}
\label{conjecture:conjecture-mirror-pair}
\(\mathcal{Y}\to V\) is mirror to \(\mathcal{Y}^{\vee}\to U\).
\end{conjecture}

Note that \(\mathcal{Y}\) and \(\mathcal{Y}^{\vee}\) are families of 
singular Calabi--Yau threefolds. The above conjecture is a generalization 
of the symmetry observed for singular \(K3\) surfaces 
\cite{2019-Hosono-Lian-Yau-k3-surfaces-from-configurations-of-six-lines-in-p2-and-mirror-symmetry-ii}*{Conjecture 6.3}.

\begin{instance}
\label{instance:1-1-1-1-2-p2}
We retain the notation in Example \ref{instance:double-cover-over-p-2} and 
\ref{instance:family-of-singular-k3-sufraces}.
Let \(\mathcal{Y}\to V\) be the gauge fixed double cover family over \(X\) along the 
nef-partition \(\{\rho_{1},\rho_{2}\}\sqcup\{\rho_{3}\}\). Equivalently, \(\mathcal{Y}\to V\)
is the family of double covers over \(X\) branched along \(4\) lines and \(1\) quadric and
\(3\) of the lines are coordinate axises.
In the present case, \(V\) is an open subset of 
\(\mathrm{H}^{0}(X,\mathscr{O}(2))^{\vee}\times\mathrm{H}^{0}(X,\mathscr{O}(1))^{\vee}\).
We denote by \([x:y:z]\) the homogeneous coordinates on \(X\). Then period integrals for
\(\mathcal{Y}\to V\) is then of the form
\begin{equation}
\label{equation:period-integral-nef-partition-11114}
\int \frac{\mathrm{d}\mu}{\sqrt{xyz(c_{1}x+c_{2}y+c_{3}z)(d_{1}x^2+d_{2}y^{2}+d_{3}z^{2}+d_{4}xy+d_{5}xz+d_{6}yz)}}
\end{equation} 
with \(\mathrm{d}\mu=x\mathrm{d}y\wedge\mathrm{d}z-y\mathrm{d}x\wedge\mathrm{d}z
+z\mathrm{d}x\wedge\mathrm{d}y\) and \(c_{i},d_{j}\in\mathbb{C}\). 
Mimicking the argument in Example \ref{instance:families-of-singular-k3-surfaces-continued}, 
we see that the period integrals
are governed by a GKZ \(A\)-hypergeometric system with
\begin{equation}
\label{equation:1-1-1-1-2-gkz}
A = 
\begin{bmatrix}
1 & 1 & 1 & 0 & 0 & 0 & 0 & 0 & 0\\
0 & 0 & 0 & 1 & 1 & 1 & 1 & 1 & 1\\
0 & 1 & 0 & 0 &-1 &-1 &-1 & 0 & 1\\
0 & 0 & 1 & 0 & 1 & 0 &-1 &-1 &-1\\
\end{bmatrix},~
\beta = 
\begin{bmatrix}
-1/2\\
-1/2\\
0\\
0\\
\end{bmatrix}.
\end{equation}

It is natural to consider the period integrals \emph{before} the gauge fixing.
Consider the family of double covers over \(X\) branched along \(4\) lines and \(1\) quadric in
general positions. Such a family can be parameterized by an open subset  
\begin{equation*}
V'\subset \mathrm{Mat}_{3\times 4}(\mathbb{C})\times \mathrm{Mat}_{6\times 1}(\mathbb{C})
=\left(\mathrm{H}^{0}(X,\mathscr{O}(1))^{\vee}\right)^{4}\times \mathrm{H}^{0}(X,\mathscr{O}(2))^{\vee}.
\end{equation*}
Precisely, the element
\begin{equation*}
\begin{bmatrix}
a_{11} & a_{12} & a_{13} & a_{14}\\
a_{21} & a_{22} & a_{23} & a_{24}\\
a_{31} & a_{32} & a_{33} & a_{34}\\
\end{bmatrix} \times
\begin{bmatrix}
b_{11}\\
b_{21}\\
b_{31}\\
b_{41}\\
b_{51}\\
b_{61}\\
\end{bmatrix}\in V'
\end{equation*}
determines a double cover branched along the lines \(a_{1i}x+a_{2i}y+a_{3i}z\)
for \(i=1,\ldots,4\) and a quadric 
\(b_{11}x^{2}+b_{21}y^{2}+b_{31}z^{2}+b_{41}xy+b_{51}xz+b_{61}yz\). 
Moving around inside \(V'\)
yields a family of double covers over \(X\). The period integrals are of the form
\begin{equation*}
\omega(\mathbf{a},\mathbf{b}):=\int \frac{\mathrm{d}\mu}
{\sqrt{\prod_{i=1}^{4}(a_{1i}x+a_{2i}y+a_{3i}z)(b_{11}x^{2}+b_{21}y^{2}+b_{31}z^{2}+b_{41}xy+b_{51}xz+b_{61}yz)}}.
\end{equation*}
It is straightforward to check that \(\omega=\omega(\mathbf{a},\mathbf{b})\) satisfies the 
system of PDEs consisting of three sets of equations 
(See Appendix \ref{section:picard-fuchs-equations-for-double-covers} for details),
which can be thought as a generalized Aomoto--Gelfand systems on 
\(\mathrm{Mat}_{3\times 4}(\mathbb{C})\times \mathrm{Mat}_{6\times 1}(\mathbb{C})\).
\end{instance}

\begin{remark} 
\label{remark:tau-ag-systems}
The system of the equations in \eqref{equation:euler-operator-generalized-am-systems}, 
\eqref{equation:symmetry-operator-generalized-am-systems}, and 
\eqref{equation:polynomial-operators-generalized-am-systems} can be identified with 
the \emph{tautological systems} defined in 
\cite{2013-Lian-Song-Yau-periodic-integrals-and-tautological-systems}
with a \emph{fractional exponent} \(\beta\). Indeed,
\eqref{equation:euler-operator-generalized-am-systems} is the Euler operator,
\eqref{equation:symmetry-operator-generalized-am-systems} is the symmetry operator generated by 
the \(\mathrm{GL}(3,\mathbb{C})\)-action on \(\mathbb{P}^{2}\), and
\eqref{equation:polynomial-operators-generalized-am-systems} is the polynomial operators determined by 
the embedding
\begin{equation}
\mathbb{P}^{2} \to \mathbb{P}^{2}\times\mathbb{P}^{2}\times\mathbb{P}^{2}\times\mathbb{P}^{2}
\times\mathbb{P}^{5}
\end{equation}
given by the line bundles \(\mathscr{O}(1),\mathscr{O}(1),\mathscr{O}(1),\mathscr{O}(1)\), and \(\mathscr{O}(2)\).

Conversely, starting with the tautological system as above, we can perform
a gauge fixing to reduce the system to a GKZ system 
given by the data in \eqref{equation:1-1-1-1-2-gkz}. One 
can then explicitly write down the secondary fan compactification, 
the unique holomorphic period near the LCSL point and the mirror map.

We also remark that even in the case of (classical) Calabi--Yau complete intersections in a 
projective homogeneous manifold \(X\) endowed with a semi-simple Lie group \(G\)-action, 
it is not clear how to write down a holomorphic series solution 
to the corresponding tautological system at the rank one point
where the existence was proven in \cites{2016-Huang-Lian-Zhu-period-integrals-and-the-riemann-hilbert-correspondence,2018-Huang-Lian-Yau-Yu-period-integrals-of-local-complete-intersections-and-tautological-systems,2020-Lee-Lian-Zhang-on-a-conjecture-of-haung-lian-yau-yu}.
\end{remark}

\appendix
\section{A proof of Theorem \ref{theorem:hodge-number-p-ne-q} and generalities on cyclic covers}
\label{section:generalities-on-cyclic-covers}
In this paragraph, we recall the construction of cyclic covers
over a smooth projective variety and investigate the Calabi--Yau condition. 
We also recall the Hodge theory needed for our cyclic covers. 
Let us fix the following notation throughout this section.
\begin{itemize}
  \item Let \(X\) be an \(n\)-dimensional smooth projective variety.
  \item Let \(L\) be a line bundle on \(X\) 
  and \(\mathscr{L}\) be the sheaf of sections of \(L\). 
  As an algebraic variety, \(L=\mathit{Spec}_{\mathscr{O}_{X}}
  (\oplus_{i=0}^\infty \mathscr{L}^{-i})\),
  where \(\mathscr{L}^{-1}\) is the dual of \(\mathscr{L}\)
  and \(\mathscr{L}^{0}:=\mathscr{O}_{X}\).
  \item Fix an integer \(r\ge 0\).
  For \(s\in\mathrm{H}^0(X,L^r)=\mathrm{Hom}_X(\mathscr{O}_X,\mathscr{L}^r)\) 
  a non-zero section,
  we denote by \(D_s\) the scheme-theoretic zero locus of \(s\). 
  \item Let \(\Omega^k_X(\log D):=\Omega^k_X(\log D_{\mathrm{red}})\) be the sheaf
  of logarithmic differential \(k\)-forms with poles along \(D\).
  \item For a normal projective variety \(Z\),
  we denote by \(\omega_Z\) the sheaf of top exterior product of the
  K\"{a}hler differential on \(Z\) and by \(K_{Z}\) the canonical divisor on \(Z\).
  Note that under the normality hypothesis on \(Z\), 
  \(\omega_{Z}\) is isomorphic to the dualizing sheaf of \(Z\).
\end{itemize}

\subsection{Basic properties}
\label{section:appendix-cyclic-covers-projective-bundle}
For \(s\in\mathrm{H}^{0}(X,L^{r})\), let 
\(Y_s':=\mathit{Spec}_{\mathscr{O}_{X}}(\mathscr{A}_s')\) 
and \(Y_{s}\) be the cyclic cover 
over \(X\) defined in \S\ref{section:cyclic-covers-general-setup}.
\begin{proposition}
\(Y_{s}'\) is smooth if and only if \(D_{s}\) is smooth.
\end{proposition}
\begin{proof}
Put \(Y':=Y_{s}'\) for simplicity. 
The question is local. Let \(\pi':Y'\to X\) be the structure morphism. 
Fix \(y\in Y'\) and put \(x=\pi'(y)\). 
We take an affine open neighborhood \(U\subset X\) of \(x\) 
over which \(L\) is trivial. 
Then 
\begin{equation*}
\left. Y'\right|_{\pi'^{-1}(U)}\simeq 
\mathscr{O}_U[y]/(y^r-f)\subset \mathbb{A}^1\times X,
\end{equation*} 
where \(f\) is the local function representing \(s\). 
If \(x_1,\ldots,x_{n}\) be a system of local coordinates around \(x\),
\(\left. Y'\right|_{\pi'^{-1}(U)}\) is singular at
\(q=(y,x)\in \left. Y'\right|_{\pi'^{-1}(U)}\) if and only if 
\begin{equation}
y^r-f(x)=0,~~ry^{r-1}=0,~\mbox{and}~\partial f/\partial x_i=0,~i=1,\ldots,n.
\end{equation}
For \(r\ge 2\), this is equivalent to \(y=0\), 
\(f(x)=0\) and \(\partial f/\partial x_i=0,~i=1,\cdots,n\), 
which means \(D_{s}\) is singular at \(x\in X\).
\end{proof}

For simplicity we put \(Y=Y_{s}\) and denote by
\(\pi:Y_{s}\to X\) the structure morphism.
The morphism \(\pi\) is \'{e}tale over \(X\setminus D_s\), with degree \(r\). 
If \(D_{s}\) is \emph{non-singular}, 
the pull-back section \(\pi^{\ast}(s)\in \mathrm{H}^0(Y,\pi^{\ast}L^r)\) 
defines a smooth subvariety \((D_{\pi^{\ast}(s)})_{\mathrm{red}}\).
According to \cite{2004-Lazarsfeld-positivity-in-algebraic-geometry-I}*{Lemma 4.2.4}, 
\(\pi^{\ast}\Omega_X^k(\log D_s)=\Omega_Y^k(\log D_{\pi^{\ast}(s)})\).
In particular, for \(k=\dim X\), we have
\begin{proposition}
\label{proposition:calabi-yau-condition}
With the same notation, we have
\begin{equation}
\pi^\ast(\omega_X\otimes\mathscr{L}^r)
\simeq\omega_Y\otimes\pi^\ast\mathscr{L}.
\end{equation}
Consequently, \(\omega_Y\simeq\mathscr{O}_Y\) if and only if
\(\pi^{\ast}(\omega_X\otimes\mathscr{L}^r)\otimes\pi^{\ast}\mathscr{L}^{-1}
\simeq\mathcal{O}_Y\).
\end{proposition}

We are mainly interested in the case when \(Y_{s}'\) or \(Y_{s}\) are singular.
To handle this situation, by compactifying the total space of \(L\), 
we regard \(Y_{s}'\) as a hypersurface 
in certain projective space bundle over \(X\). 
To explain this in more detail, let us recall the construction of
projective bundle spaces over \(X\).

Let \(\mathscr{E}\) be a locally free sheaf of rank \((m+1)\) over \(X\)
and \(Z:=\mathit{Proj}_{\mathscr{O}_{X}}(\mathrm{Sym}^{\bullet}(\mathscr{E}))\) 
be the associated projective space bundle.
We denote by \(\eta:Z\to X\) the structure morphism.
We have the relative Euler sequence
\begin{equation}
\label{diagram:relative-euler-sequence-for-projective-space-bundles}
0\to \Omega_{Z/X}\to \eta^\ast\mathscr{E}(-1)\to\mathscr{O}_Z\to 0.
\end{equation}
Here \(\eta^{\ast}\mathscr{E}(-1)=\eta^{\ast}\mathscr{E}
\otimes\mathscr{O}_{Z/X}(1)^\vee\) and \(\mathscr{O}_{Z/X}(1)\) 
is the relative ample sheaf.
Taking exterior products yields \(\omega_{Z/X}\simeq \eta^{\ast}
(\wedge^{m+1}\mathscr{E})(-m-1)\).   

Given \(X,L,s,r\) as above, let \(Y'=Y_{s}'\) and \(Y=Y_{s}\) as before.
We consider the rank two bundle
\(\mathscr{E}:=\mathscr{O}_X\oplus\mathscr{L}^{-1}\)
and the associated projective space bundle \(\eta:Z\to X\).
Note that \(\wedge^2\mathscr{E}\simeq\mathscr{L}^{-1}\) and therefore
\begin{equation*}
\omega_Z\simeq \omega_{Z/X}\otimes\eta^\ast\omega_X
\simeq \eta^\ast\mathscr{L}^{-1}\otimes\eta^\ast\omega_X\otimes\mathscr{O}_{Z/X}(-2).
\end{equation*}
\(Y'\) can be regard as a hypersurface in \(Z\). 
Since \(\mathscr{O}_{Z/X}(-2)\) is trivial over \(Y'\), 
the dualizing sheaf \(\omega_{Y'}\) is trivial 
if and only if \(\mathscr{L}^{r}\simeq\mathscr{L}\otimes\omega_X^{-1}\).
\begin{remark}
\(Y'\) may \emph{not} be an anti-canonical hypersurface in \(Z\).
(Indeed, it is never the case unless \( r = 2 \)).
\end{remark}
From the viewpoint of hyperplane sections, 
we obtain 
\begin{proposition}
\label{proposition:calabi-yau-condition-singular}
\(Y'\) is Cohen--Macaulay. Furthermore,
if \(\mathrm{codim}_X\mathrm{Sing}(D_s)\ge 2\), then \(Y'\) is normal and \(Y=Y'\).
In this case \(\omega_{Y}\simeq \mathscr{O}_{Y}\) if and only if
\begin{equation*}
\omega_X\otimes\mathscr{L}^{r-1}\simeq\mathscr{O}_X.
\end{equation*} 
\end{proposition}
\begin{proof}
Since \(Z\) is smooth, it is Cohen--Macaulay. The first statement is clear.
Now suppose \(\mathrm{codim}_X\mathrm{Sing}(D_s)\ge 2\).
Then \(Y'\) is regular in codimension one and hence, by Serre's criterion, \(Y'\) is normal
and \(Y=Y'\).
It then follows that the canonical sheaf of \(Y\) (the top exterior power of
the K\"{a}hler differential) is
isomorphic to the dualizing sheaf \(\omega_Y\), 
and the later one is locally free by adjunction formula. 
In particular, the canonical sheaf of \(Y\) is Cartier.
We compute 
\begin{equation}
\begin{split}
\omega_Y&\simeq \left.\omega_Z\otimes \mathscr{O}_Z(Y)\right|_Y\\
&\simeq \left.\mathscr{O}_Z(r-2)\otimes
\eta^\ast\omega_X\otimes\eta^\ast\mathscr{L}^{r-1}\right|_Y\\
&\simeq \left.\eta^\ast(\omega_X\otimes\mathscr{L}^{r-1})\right|_Y\\
&\simeq \pi^{\ast}(\omega_X\otimes\mathscr{L}^{r-1}),
\end{split}
\end{equation}
where \(\pi=\left.\eta\right|_{Y}\). If \(\omega_{Y}\simeq\mathscr{O}_{Y}\), 
then, by projection formula, 
\begin{equation*}
(\omega_X\otimes\mathscr{L}^{r-1})\otimes\pi_{\ast}\mathscr{O}_{Y}\simeq
\pi_{\ast}\mathscr{O}_{Y}.
\end{equation*}
Since the isomorphism respects the \(\mathbb{Z}\slash r\mathbb{Z}\)-action,
from the eigenspace decompositions, it follows \(\omega_X\otimes\mathscr{L}^{r-1}\simeq\mathscr{O}_{X}\).
\end{proof}

\begin{remark}
\label{remark:smooth-cyclic-cover-mirror}
This interpretation allows us to reduce the general smooth cyclic covers to 
the case of classical hypersurfaces. We will discuss it in Appendix \ref{section:mirror-symmetry-of-smooth-cyclic-covers}.
\end{remark}

\subsection{Hodge numbers}
We review some basic facts about
the Hodge theory for orbifolds proved in
\cites{1957-Baily-on-the-imbedding-of-v-manifolds-in-projective-space,1977-Steenbrink-mixed-hodge-structure-on-the-vanishing-cohomology} and 
\cite{2012-Arapura-hodge-theory-of-cyclic-covers-branched-over-a-union-of-hyperplanes}*{\S1},
which are applicable to the case of cyclic covers over a smooth manifold.

Let \(s\in\mathrm{H}^0(X,L^{r})\) with \(D:=D_{s}\) being a simple normal crossing divisor.
In particular, \(\mathrm{codim}_X\mathrm{Sing}(D)\ge 2\).
Let \(\pi:Y\to X\) be the cyclic cover. 
\(Y\) is smooth outside \(\pi^{-1}(\mathrm{Sing}(D))\).
Denote by \(Y^{\mathrm{reg}}\) the non-singular part of \(Y\) and \(j:Y^{\mathrm{reg}}\to Y\).
In \cite{1977-Steenbrink-mixed-hodge-structure-on-the-vanishing-cohomology},
Steenbrink defined \(\tilde{\Omega}_{Y}^{k}:=j_{\ast}\Omega_{Y^{\mathrm{reg}}}^{k}\) and proved that
\begin{itemize}
  \item[(a)] There is a canonical, purely weight \(k\) Hodge structure on 
  \( \mathrm{H}^{k}(Y,\mathbb{Q}) \).
  (cf.~\cite{1977-Steenbrink-mixed-hodge-structure-on-the-vanishing-cohomology}*{Corollary 1.5}).
  \item[(b)] There is a spectral sequence for hypercohomology groups
  \begin{equation*}
    \mathrm{H}^{q}(Y,\tilde{\Omega}_{Y}^{p}) 
    \Rightarrow \mathbb{H}^{p+q}(Y,\tilde{\Omega}_{Y}^{\bullet})=\mathrm{H}^{p+q}(Y,\mathbb{C}).
  \end{equation*}
  (cf.~\cite{1977-Steenbrink-mixed-hodge-structure-on-the-vanishing-cohomology}*{Theorem 1.12}).
  \item[(c)] The hard Lefschetz theorem holds for \(Y\). 
  (cf.~\cite{1977-Steenbrink-mixed-hodge-structure-on-the-vanishing-cohomology}*{Theorem 1.13}).
\end{itemize}
In addition, as observed by Arapura in \cite{2012-Arapura-hodge-theory-of-cyclic-covers-branched-over-a-union-of-hyperplanes},
we have
\begin{itemize}
  \item[(d)] There is an isomorphism
  \begin{equation*}
  \mathrm{H}^{k}(Y-E,\mathbb{C})\simeq \bigoplus_{p+q=k} \mathrm{H}^{q}(Y,\tilde{\Omega}^{p}_{Y}(\log E)),
  \end{equation*}
  where \(\tilde{\Omega}^{p}_{Y}(\log E):=j_{\ast}\Omega^{p}_{Y^{\mathrm{reg}}}(\log E\cap Y^{\mathrm{reg}})\)
  and \(E:=\pi^{-1}(D)\).
\end{itemize}
Since \(\pi\) is finite, 
\begin{equation*}
\mathrm{H}^{q}(Y,\tilde{\Omega}_{Y}^{p})\simeq \mathrm{H}^{q}(X,\pi_{\ast}\tilde{\Omega}_{Y}^{p}).
\end{equation*}
Since \(Y\) is normal, we have \( (\Omega_{Y}^{p})^{\vee\vee}\simeq j_{\ast}\Omega_{Y^{\mathrm{reg}}}^{p} = \tilde{\Omega}_{Y}^{p} \).
Most statements in \cite{1992-Esnault-Viehweg-lectures-on-vanishing-theorems}*{Lemma 3.16} can be extended to our case.

\begin{proposition}[See also \cite{2012-Arapura-hodge-theory-of-cyclic-covers-branched-over-a-union-of-hyperplanes}*{Lemma 1.5}]
\label{proposition:decomposition-pushforward-log-differential}
Let \(L\) be an ample line bundle and \(s\in\mathrm{H}^0(X,L^r)\).
Assume that $D:=D_s$ is a simple normal crossing divisor. Then we have
\begin{equation*}
\label{equation:log-differential-pushforward-log-defferential}
\pi_{\ast}\tilde{\Omega}_{Y}^{p}(\log E)\simeq \bigoplus_{i=0}^{r-1} 
\Omega_X^p(\log D)\otimes\mathscr{L}^{-i}~\mbox{and}~
\pi_{\ast}\tilde{\Omega}_{Y}^{p}\simeq \bigoplus_{i=0}^{r-1} 
\Omega_X^p(\log D^{(i)})\otimes\mathscr{L}^{-i},
\end{equation*}
where \(D^{(i)}=D\) for \(i\ne 0\) and \(D^{(0)}=0\).
Furthermore, if $p+q\ne n$, then we have 
\(\mathrm{H}^q(\Omega_X^p(\log D) \otimes\mathscr{L}^{-i})=0\) 
for all \(i\ne 0\). Consequently, 
\begin{equation}
\label{equation:Hodge-numbers}
h^{p,q}(X,\mathbb{C})= h^{p,q}(Y,\mathbb{C}),~\mbox{for}~p+q\ne n.
\end{equation}
\end{proposition}
\begin{proof}
Look at the fibred diagram
\begin{equation*}
\begin{tikzcd}
  & Y^{\mathrm{reg}} \ar[r,"j"]\ar[d] & Y \ar[d,"\pi"]\\
  & X\setminus D_{\mathrm{sing}} \ar[r,] &X.
\end{tikzcd}
\end{equation*}
The correpsonding pushforward formulae hold
for \(Y^{\mathrm{reg}} \to X\setminus D_{\mathrm{sing}}\) by 
\cite{1992-Esnault-Viehweg-lectures-on-vanishing-theorems}*{Lemma 3.16(a)(d)}.
Pushing forward the equality via \(j\), we obtain \rm{(i)} and \rm{(ii)} 
since both sides of them are reflexive sheaves.

For the second part, we observe that, for $E:=\pi^{-1}(D)$,
\begin{equation*}
\label{equation:logarithmic-open-affine-inclusion}
\mathrm{H}^q(\Omega_X^p(\log D)\otimes \mathscr{L}^{-i})
\subset \mathrm{H}^{p+q}(Y-E,\mathbb{C})=0,
~\mbox{as}~p+q>n,
\end{equation*}
by the affine vanishing theorem 
\cite{1986-Esnault-Viehweg-logarithmic-de-rham-complexes-and-vanishing-theorems}*{Corollary 1.5}. 
The statement for $p+q<n$ follows from the hard Lefschetz on $Y$ and a duality argument.
\end{proof}

\begin{remark}
For non-ample \(L\), these statements still hold if \(D\) is a simple normal crossing divisor
such that \(X\setminus D\) is affine.
\end{remark}

\section{A mirror construction of smooth cyclic covers}
\label{section:mirror-symmetry-of-smooth-cyclic-covers}

In this paragraph, we explain how to construct the (topological) mirror Calabi--Yau
family when \(Y=Y'\) is a \emph{smooth} cyclic cover \(X\)
(for notation, see \S\ref{section:appendix-cyclic-covers-projective-bundle}).

Let us outline the procedure. As we noted in \S\ref{section:appendix-cyclic-covers-projective-bundle},
we can compactify the total space of \(L\to X\) by 
\(Z:=\mathit{Proj}_{\mathscr{O}_{X}}(\mathrm{Sym}^{\bullet}\mathscr{E})\) 
and \(Y\) can be realized as a hypersurface in \(Z\).
Note that \(Z\) is a smooth semi-Fano toric variety. 
However, \(Y\) may \emph{not} be an anti-canonical hypersurface in \(Z\).
To remedy this defect, we will construct a 
contraction \(\phi\colon Z\to Z'\) such that \(Y\to \phi(Y)\) is a crepant resolution
and \(\phi(Y)\) is an anti-canonical hypersurface in \(Z'\). Now
the Batyrev's duality construction applies.
We keep the notation in \S\ref{subsection:B-B-duality-construction}.
\subsection{Projective bundle spaces and its toric contraction}
\label{subsection:contraction-toric}
Let \(X\) be a smooth semi-Fano toric variety and \(Y\to X\) be a \emph{smooth} cyclic \(r\)-fold
cover over \(X\). Let \(\mathscr{L}\) be a big and nef line bundle on \(X\).
We fix a torus invariant divisor \(D=\sum_{j=1}^{p}a_{j}D_{j}\) with \(a_{j}\ge 0\)
such that \(\mathscr{L}\simeq\mathscr{O}_{X}(D)\).
Put \(\mathscr{E}=\mathscr{O}_X\oplus\mathscr{L}^{\vee}\)
and let \(Z=\mathit{Proj}_X(\mathrm{Sym}^{\bullet}\mathscr{E})=\mathbf{P}_{X}(L\oplus\mathbb{C})\).
This is a toric variety and we now describe its toric data.

Let \(\bar{N}:=N\times\mathbb{Z}\).
Denote by \(\mathrm{e}_{\infty}:=(\mathbf{0},1)\) and \(\mathrm{e}_{0}:=(\mathrm{0},-1)\). 
Consider 
\begin{align*}
\begin{split}
\mathcal{S}_1&:=\{\bar{\rho}_j=\rho_j+a_j\mathrm{e}_{\infty}~|~j=1,\ldots,p\},~\mbox{and}\\
\mathcal{S}_2&:=\{\mathrm{e}_{\infty},\mathrm{e}_{0}\}.
\end{split}
\end{align*}
Any maximal cone \(\tau\in\Sigma(n):=\Sigma_{X}(n)\) determines two maximal 
cones in \(\bar{N}\):
\begin{align}\label{maximal:cones-liftings}
\begin{split}
\tau_{0}&=\mathrm{Cone}(\{\bar{\rho}_j~|~\rho_j\in\tau(1)\}\cup\{\mathrm{e}_{0}\}),~\mbox{and}~\\
\tau_{\infty}&=\mathrm{Cone}(\{\bar{\rho}_j~|~\rho_j\in\tau(1)\}\cup\{\mathrm{e}_{\infty}\}).
\end{split}
\end{align}
\begin{definition}
Let \(\Sigma_{Z}\) be the collection of \(\tau_{0}\) and \(\tau_{\infty}\) 
as well as all their faces for all \(\tau\in\Sigma(n)\). 
The following proposition is straightforward.
\end{definition}
\begin{proposition}
\(\Sigma_{Z}\) is a fan and defines the toric variety \(Z\). Furthermore, 
from the construction, the infinity 
divisor is given by the \(1\)-cone \(\mathbb{R}_{\ge 0}\cdot\mathrm{e}_{\infty}\).
\end{proposition}

Note that the canonical bundle \(\omega_Z\) is isomorphic to \(\mathscr{O}_{Z/X}(-2)\otimes 
\eta^{\ast}\mathscr{L}^{\vee}\otimes \eta^{\ast}\omega_{X}\), 
where \(\eta:Z\to X\) is the structure morphism and \(\mathscr{O}_{Z/X}(1)\)
is the relative ample sheaf. 
It is easy to check that 
\(\mathscr{O}_{Z/X}(1)\simeq\mathscr{O}_Z(D_{\mathrm{e}_{\infty}})\),
\(Y\) is a section of \(\mathscr{O}_{Z/X}(r)\otimes \eta^{\ast}\mathscr{L}^{r}\) 
and 
\begin{equation}
D_{\mathrm{e}_{0}}\sim \sum_{j=1}^p a_jD_{\bar\rho_j}+D_{\mathrm{e}_{\infty}}~\mbox{on}~Z.
\end{equation}

\begin{proposition}
\label{proposition:H-is-nef}
The divisor \(H:=D_{\mathrm{e}_{\infty}}+
\sum_{j=1}^p a_j D_{\bar\rho_j}\) is base point free. 
\end{proposition}
\begin{proof}
We only have to show that the divisor \(H\) is 
numerically effective, which implies base point free in toric cases 
\cite{2011-Cox-Little-Schenck-toric-varieties}*{Theorem 6.3.12}.
We can prove this by using the notion of primitive collections and corresponding curves
on toric varieties. We leave the details to the reader.
\end{proof}

\begin{remark}
When \(a_{j}=1\) for all \(j\), we see that
\(2H\) is linearly equivalent to \(-K_{Z}\) and \(Y\)
is an anti-canonical hypersurface in \(Z\).
\end{remark}

\subsection{The Calabi--Yau condition}
Now we impose the Calabi--Yau condition for the cyclic cover \(Y\);
namely \(\mathscr{L}^{r-1}\simeq\omega_{X}^{-1}\)
(cf.~Proposition \ref{proposition:calabi-yau-condition-singular}).
This automatically implies that \(\mathscr{L}\) is big and nef.
Under this assumption,
we see that \(Z\) is semi-Fano.

Let us recall the construction in 
\cite{2000-Mavlyutov-semi-ample-hypersurfaces-in-toric-varieties}.
Let \(X=X_{\Sigma}\) be an \(n\)-dimensional complete toric variety and \(H\) be an \(n\)-semiample divisor. 
Recall that a Cartier divisor \(H\) is \(n\)-semiample if \(H\) is
generated by global sections and \(H^{n}>0\) where \(n=\dim X\),
or equivalently, \(H\) is generated by global sections and \(\Delta_{H}\)
is of maximal dimension \(n\), or equivalently \(\mathscr{O}_{X}(H)\) is big and nef.
Assume that \(H=\sum_{\rho\in\Sigma(1)} a_{\rho}D_{\rho}\). 
We denote by \(\psi_{H}\) the support function associated with \(H\).
In the present case, \(\psi_{H}\) is convex.
For each \(\sigma\in \Sigma(n)\), we can find an element
\(m_{\sigma}\in M\) such that 
\begin{equation*}
\psi_{H}(u)=\langle u,m_{\sigma}\rangle,~u\in \sigma.
\end{equation*}
The collection \(\{m_{\sigma}\}_{\sigma\in\Sigma(n)}\) is 
called the \emph{Cartier data} of \(H\).
We glue together those maximal dimensional cones in \(\Sigma\) 
having the same \(m_{\sigma}\) and obtain a convex
rational polyhedral cone. In the present case, these cones are 
in fact strongly convex since \(\Delta_H\) has maximal dimension \(n\).
The set of these strongly convex rational polyhedral cones gives rise
to a new fan \(\Sigma_{H}\).
We remark that for each \(r\in\mathbb{Q}_{>0}\), \(rH\) produces the same fan.
Moreover, the fan \(\Sigma\) is a subdivision of \(\Sigma_{H}\). Let 
\(\pi\colon X\to X_{\Sigma_{H}}\) be the corresponding 
toric morphism and
\(\pi_{\ast}\colon A_{n-1}(X)\to A_{n-1}(X_{\Sigma_{H}})\)
be the pushforward map between Chow groups.
We have the following proposition (cf.~
\cite{2000-Mavlyutov-semi-ample-hypersurfaces-in-toric-varieties}*{Proposition 1.2}).
\begin{proposition}
Let \(X=X_{\Sigma}\) and \(H\) be an \(n\)-semiample divisor. Then
there exists a unique complete 
toric variety \(X_{\Sigma_{H}}\) 
with a toric birational map \(\pi\colon X_{\Sigma}\to X_{\Sigma_{H}}\) such that
\(\Sigma\) is a refinement of \(\Sigma_{H}\), \(\pi_{\ast}[H]\),
is ample, and \(\pi^{\ast}\pi_{\ast}[H]=[H]\). Moreover,
\(\Sigma_{H}\) is the normal fan of \(\Delta_{H}\);
in other words, \(\mathbf{P}_{\Delta_{H}}=X_{\Sigma_{H}}\).
\end{proposition}

For simplicity, we put \(X'=\mathbf{P}_{\Delta}\).
Recall that \(f\colon X\to X'\) is an MPCP desingularization.
We have \(f^{\ast}\omega_{X'}^{-1}\simeq \omega_{X}^{-1}\)
with \(\omega_{X'}^{-1}\) ample.
The Calabi--Yau condition \((r-1)D=-K_{X}\), where
\(D=\sum_{j=1}^{p} a_{j}D_{\rho_{j}}\), implies that
\(D\) is also \(n\)-semiample.
By the previous proposition, \(D'=f_{\ast}[D]\) is ample on \(X'\) and
\begin{eqnarray*}
(r-1) f_{\ast}[D]=(r-1)D'\sim -K_{X'}.
\end{eqnarray*}

We give a construction of the contraction \(\phi\colon Z\to Z'\). 
We observe that \(\Delta_{H}\) for the nef divisor \(H=D_{\mathrm{e}_{\infty}}+
\sum_{j=1}^p a_j D_{\bar\rho_j}\) is of maximal dimension and hence \(H\) is \((n+1)\)-semiample.
The Cartier data of \(H\) is easy to describe.
\begin{lemma}
Let \(\{m_{\sigma}\}_{\sigma\in\Sigma(n)}\) be the Cartier data for \(D\).
Then \(\bar{m}_{\tau_{\infty}}:=(0,1)\)
and \(\bar{m}_{\tau_{0}}:=(m_{\tau},0)\) for \(\tau\in\Sigma(n)\) 
give the Cartier data of \(H\).
\end{lemma}
It follows from the construction in 
\cite{2000-Mavlyutov-semi-ample-hypersurfaces-in-toric-varieties}*{Proposition~1.2}
that there exists a toric map \(\phi\colon Z\to Z':=\mathbf{P}_{\Delta_{H}}\),
where \(\Delta_{H}\) is the polytope of \(H\).
Moreover, \(H'=\phi_{\ast}[H]\) is an ample divisor on \(Z'\) such that 
\(\phi^{\ast} \phi_{\ast}[H] = \mathscr{O}_{Z}(H)\).
It is straightforward to see that \(Z'\) is obtained by
contracting the infinity divisor in \(\mathbf{P}_{X'}(L'\oplus\mathbb{C})\),
where \(L'\) is the geometric line bundle of \(D'\).
Such a contraction exists since \(D'\) is ample.

\begin{proposition}
\(Z'\) is Fano.
\end{proposition}
\begin{proof}
This follows from the fact that \(rH'\simeq -K_{Z'}\).
\end{proof}

\begin{corollary}
The image hypersurface \(\phi(Y)\) is in the anti-canonical class.
\end{corollary}
\begin{proof}
The result follows since \(\phi^{\ast}\omega_{Z'}^{-1}
\simeq\phi^{\ast}\mathscr{O}_{Z'}(rH')=\mathscr{O}_{Z}(rH)\).
\end{proof}

\subsection{The mirror construction}
Now we begin with a tuple \((X,L,s,r)\) satisfying the Calabi--Yau condition 
for some \(r\ge 2\) and \(Y\) is the \(r\)-fold cyclic cover as before. 
We have constructed \(Z\), \(Z'\) and \(\phi\colon Z\to Z'\) 
where \(Z'\) is a Fano toric variety.
Consider the restriction \(\left.\phi\right|_{Y}\colon Y \to \phi(Y)\).
We see that \(\phi(Y)\) is an anti-canonical hypersurface in \(Z'\) and
\(\left.\phi\right|_{Y}\) is a crepant resolution of \(\phi(Y)\).
Now since \(Z'\) is Fano, we can apply Batyrev's duality construction to
obtain the mirror family.

\newpage
\section{Picard--Fuchs equations for double covers}
\label{section:picard-fuchs-equations-for-double-covers}
In this paragraph, we list the equations in the PDE systems which 
govern the period integrals \eqref{equation:period-integral-nef-partition-11114}.

\begin{align}
\label{equation:euler-operator-generalized-am-systems}
\begin{split}
&\sum_{i=1}^{6} b_{i1}\frac{\partial}{\partial b_{i1}} 
\omega = -\frac{1}{2}\omega,\\
&\sum_{i=1}^{3} a_{ij}\frac{\partial}{\partial a_{ij}} 
\omega = -\frac{1}{2}\omega,~j=1,\ldots,4.
\end{split}
\end{align}

\begin{align}
\label{equation:symmetry-operator-generalized-am-systems}
\begin{split}
&\left(\sum_{k=1}^{4}a_{1k}\frac{\partial}{\partial a_{2k}}+2b_{11}\frac{\partial}{\partial b_{41}}
+b_{41}\frac{\partial}{\partial b_{21}}+b_{51}\frac{\partial}{\partial b_{61}}\right)
\omega=0,\\
&\left(\sum_{k=1}^{4}a_{1k}\frac{\partial}{\partial a_{3k}}+2b_{11}\frac{\partial}{\partial b_{51}}
+b_{41}\frac{\partial}{\partial b_{61}}+b_{51}\frac{\partial}{\partial b_{31}}\right)
\omega=0,\\
&\left(\sum_{k=1}^{4}a_{2k}\frac{\partial}{\partial a_{3k}}+2b_{21}\frac{\partial}{\partial b_{61}}
+b_{41}\frac{\partial}{\partial b_{51}}+b_{61}\frac{\partial}{\partial b_{31}}\right)
\omega=0,\\
&\left(\sum_{k=1}^{4}a_{2k}\frac{\partial}{\partial a_{1k}}+2b_{21}\frac{\partial}{\partial b_{41}}
+b_{41}\frac{\partial}{\partial b_{11}}+b_{61}\frac{\partial}{\partial b_{51}}\right)
\omega=0,\\
&\left(\sum_{k=1}^{4}a_{3k}\frac{\partial}{\partial a_{1k}}+2b_{31}\frac{\partial}{\partial b_{51}}
+b_{51}\frac{\partial}{\partial b_{11}}+b_{61}\frac{\partial}{\partial b_{41}}\right)
\omega=0,\\
&\left(\sum_{k=1}^{4}a_{3k}\frac{\partial}{\partial a_{2k}}+2b_{31}\frac{\partial}{\partial b_{61}}
+b_{51}\frac{\partial}{\partial b_{41}}+b_{61}\frac{\partial}{\partial b_{21}}\right)
\omega=0,\\
&\left(\sum_{k=1}^{4}a_{1k}\frac{\partial}{\partial a_{1k}}+2b_{11}\frac{\partial}{\partial b_{11}}
+b_{41}\frac{\partial}{\partial b_{41}}+b_{51}\frac{\partial}{\partial b_{51}}+1\right)
\omega=0,\\
&\left(\sum_{k=1}^{4}a_{2k}\frac{\partial}{\partial a_{2k}}+2b_{21}\frac{\partial}{\partial b_{21}}
+b_{41}\frac{\partial}{\partial b_{41}}+b_{61}\frac{\partial}{\partial b_{61}}+1\right)
\omega=0,\\
&\left(\sum_{k=1}^{4}a_{3k}\frac{\partial}{\partial a_{3k}}+2b_{31}\frac{\partial}{\partial b_{31}}
+b_{51}\frac{\partial}{\partial b_{51}}+b_{61}\frac{\partial}{\partial b_{61}}+1\right)
\omega=0.\\
\end{split}
\end{align}

\begin{align}
\label{equation:polynomial-operators-generalized-am-systems}
\begin{split}
&\left(\frac{\partial^{2}}{\partial a_{ij}\partial a_{kl}} - \frac{\partial^{2}}{\partial a_{il}\partial a_{kj}}\right)
\omega=0,~1\le i,k\le 3,~1\le j,l\le 4.\\
&\left(\frac{\partial^{2}}{\partial b_{11}\partial b_{21}} - \frac{\partial^{2}}{\partial b_{41}^{2}}\right)
\omega=0,~
\left(\frac{\partial^{2}}{\partial b_{11}\partial b_{31}} - \frac{\partial^{2}}{\partial b_{51}^{2}}\right)
\omega=0,\\
&\left(\frac{\partial^{2}}{\partial b_{21}\partial b_{31}} - \frac{\partial^{2}}{\partial b_{61}^{2}}\right)
\omega=0,~
\left(\frac{\partial^{2}}{\partial b_{11}\partial b_{61}} - \frac{\partial^{2}}{\partial b_{41}\partial b_{51}}\right)
\omega=0,\\
&\left(\frac{\partial^{2}}{\partial b_{21}\partial b_{51}} - \frac{\partial^{2}}{\partial b_{41}\partial b_{61}}\right)
\omega=0,~
\left(\frac{\partial^{2}}{\partial b_{31}\partial b_{41}} - \frac{\partial^{2}}{\partial b_{51}\partial b_{61}}\right)
\omega=0,\\
&\left(\frac{\partial^{2}}{\partial a_{11}\partial b_{21}} - \frac{\partial^{2}}{\partial a_{21}\partial b_{41}}\right)
\omega=0,~
\left(\frac{\partial^{2}}{\partial a_{11}\partial b_{31}} - \frac{\partial^{2}}{\partial a_{31}\partial b_{51}}\right)
\omega=0,\\
&\left(\frac{\partial^{2}}{\partial a_{11}\partial b_{41}} - \frac{\partial^{2}}{\partial a_{21}\partial b_{11}}\right)
\omega=0,~
\left(\frac{\partial^{2}}{\partial a_{11}\partial b_{51}} - \frac{\partial^{2}}{\partial a_{31}\partial b_{11}}\right)
\omega=0,\\
&\left(\frac{\partial^{2}}{\partial a_{11}\partial b_{61}} - \frac{\partial^{2}}{\partial a_{21}\partial b_{51}}\right)
\omega=0,~
\left(\frac{\partial^{2}}{\partial a_{11}\partial b_{61}} - \frac{\partial^{2}}{\partial a_{31}\partial b_{41}}\right)
\omega=0,\\
&\left(\frac{\partial^{2}}{\partial a_{21}\partial b_{31}} - \frac{\partial^{2}}{\partial a_{31}\partial b_{61}}\right)
\omega=0.
\end{split}
\end{align}


\end{document}